\documentclass[12pt]{amsart}
\usepackage[english]{babel}
\usepackage{amsmath,amsthm,amsfonts,amssymb,epsfig,color,ulem,hyperref}
\usepackage[left=1in,top=1in,right=1in]{geometry}
\usepackage{enumitem}

\newtheorem{thm}{Theorem}[section]
\newtheorem{lem}[thm]{Lemma}
\newtheorem{prop}[thm]{Proposition}

\newtheorem{conj}[thm]{Conjecture}

\theoremstyle{definition}

\newtheorem{assn}[thm]{Assumption}
\hypersetup{colorlinks=true, linkcolor=blue, urlcolor=blue}
 
\newcommand{\E}{\mathbb{E}}
\newcommand{\PP}{\mathbb{P}}
\newcommand{\R}{\mathbb{R}}
\newcommand{\Z}{\mathbb{Z}}

\newcommand{\e}{\varepsilon}
\newcommand{\F}{\mathcal F}
\newcommand{\rd}{{\rm d} }
\newcommand{\lpl}{(\Lambda \cup \partial \Lambda)}

\definecolor{aocolor}{rgb}{0.0, 0.5, 0.0}

\DeclareMathOperator*{\argmin}{argmin}
\DeclareMathOperator*{\sign}{\mathrm{sign}}

\begin{document}

\title[]{Absence of Disorder Chaos for Ising Spin Glasses on $\Z^d$}

\author[L.-P. Arguin]{L.-P. Arguin}            
 \address{L.-P. Arguin\\ 
 Department of Mathematics\\
 City University of New York, Baruch College and Graduate Center\\
 New York, NY 10010}
\thanks{The research of L.-P. A. is supported in part by NSF CAREER~DMS-1653602.}
\email{louis-pierre.arguin@baruch.cuny.edu}

\author[J. Hanson]{J. Hanson}
\address{J. Hanson\\
Department of Mathematics\\
 City University of New York, City College\\
 New York, NY 10031
}            

\thanks{The research of J.H. is supported in part by NSF~DMS-1612921 .}
\email{jhanson@ccny.cuny.edu}

\date{September 19, 2019}


\maketitle

\begin{abstract}
We identify simple mechanisms that prevent the onset of disorder chaos for the Ising spin glass model on $\Z^d$. This was first shown by Chatterjee in the case of Gaussian couplings. We present three proofs of the theorem for general couplings with continuous distribution based on the presence in the coupling realization of stabilizing features of positive density. 
\end{abstract}

\section{Introduction}

\subsection{Main Result}
Consider a square box $\Lambda$ in $\Z^d$ with edge set $\Lambda^*$ and exterior vertex boundary $\partial \Lambda$. The Hamiltonian of the Ising spin glass, or Edwards-Anderson model, is
\begin{equation}
\label{eqn: Hamilton}
H_{\Lambda,J}(s)=\sum_{\{x,y\}\in (\Lambda\cup \partial \Lambda)^*} -J_{xy}s_xs_y, \quad s\in \{-1,+1\}^\Lambda\ ,
\end{equation}
where the couplings $J=(J_{x,y},\{x,y\}\in (\Lambda \cup \partial \Lambda)^*)$ are IID random variables under some probability $\PP$. The distribution of the couplings is usually taken to be symmetric, but this will not be necessary for the proofs. 
The choice of boundary condition corresponds to setting $s$ on $\partial \Lambda$. This choice will not play a role in the result.

The ground state $\sigma(J)$ at a realization $J$ of the coupling is the minimizer of $H_{\Lambda,J}$:
$$
\sigma(J)=\argmin_{s\in \{-1,+1\}^\Lambda}H_{\Lambda,J}(s)\ .
$$
This implies that the flip of spins in any subset $\mathcal B\subseteq \Lambda$ must increase the energy, yielding the equivalent characterization of the ground state:
\begin{equation}
\label{eqn: GS prop}
\sum_{\{x,y\}\in \partial \mathcal B}J_{xy}\sigma_x(J)\sigma_y(J)>0 \qquad \forall \mathcal B\subseteq \Lambda\ ,
\end{equation}
where $\{x,y\} \in \partial \mathcal{B}$ means $x \in \mathcal{B}$ and $y \notin \mathcal{B}$ (or vice-versa).

In the case where the Hamiltonian admits a global spin symmetry (e.g., with periodic boundary conditions), there is a trivial degeneracy for the ground state. 
One can then work on spin configurations modulo the spin flip, and speak of ground state pair. 
Since the arguments presented below are identical in this framework, we will omit the distinction in the notation. 

There are also non-trivial degeneracies at some special values of the couplings corresponding to values where one subset has a zero flip energy, i.e., the left-hand side of Equation \eqref{eqn: GS prop} is $0$. The set of these critical values is given by 
\begin{equation}
\label{eqn: critical set}
\mathcal C=\bigcup_{s,s'\in \{-1,+1\}^{\Lambda}, s\neq s'} \Big\{J: \sum_{\{x,y\}\in (\Lambda\cup \partial \Lambda)^*} J_{xy}(s_xs_y-s'_xs'_y)=0\Big\}\ .
\end{equation}
The ground state is well defined on the open set $\R^{(\Lambda \cup \partial \Lambda)^*}\setminus \mathcal C$.
Note that by the continuity of the distribution we have $\PP(\mathcal C)=0$.

The phenomenon of disorder chaos in spin glasses was proposed in the physics literature in \cite{fisher-huse} and in \cite{bray-moore}.
Roughly speaking, the model exhibits disorder chaos if the ground state at $J$ and the one at a value very close to $J$ differ substantially. 
To make this precise, consider the overlap
\begin{equation}
Q_\Lambda(\sigma,\sigma')=\frac{1}{|\Lambda^*|}\sum_{\{x,y\}\in \Lambda^*}\sigma_x\sigma_y\sigma'_x\sigma'_y\ .
\end{equation}
For the perturbations, we also consider two IID copies $\e,\e'$ of a continuous random variable on $\PP$ independent of $J$ and having mean $0$, and a parameter $t\geq0$ controlling the magnitude.
The main result is a proof of absence of disorder chaos in the sense that the average of the overlap between two ground states with slightly different couplings is bounded away from $0$ uniformly in $\Lambda$. \begin{thm}\label{thm: main}
For any $\delta>0$, there exists $t_0=t_0(\delta)>0$ such that 
$$
\E[Q_\Lambda(\sigma(J+t\e),\sigma(J+t\e'))]>(1-\delta) \PP(A)\\ , \text{ for any $t<t_0(\delta)$,}
$$
where $A$ is some explicit event with $\PP(A)>0$ uniformly in $\Lambda$.
\end{thm}
The theorem was first proved by Chatterjee \cite{chatterjee_2014} in the case of Gaussian couplings with an explicit decay in $t$.
Namely, if $J(t)$ and $J'(t)$ are two Ornstein-Uhlenbeck processes both starting at $J(0)$ then evolving independently, he proved for some $c>0$ that
\begin{equation}
\label{eqn: chat}
\E[Q_\Lambda(\sigma(J(t)),\sigma(J'(t)))]\geq \frac{c}{4d^2} e^{-t/(4d^2 c)}\ .
\end{equation}
The result is to be compared to the Sherrington-Kirkpatrick model on the complete graph with Gaussian coupling for which the average overlap goes to $0$ as $\Lambda\to \Z^d$ for any fixed $t$ \cite{chatterjee_2014} (the proof there is given at positive temperature, but the result is expected to hold at zero temperature as well).
The proof of \eqref{eqn: chat} is done in two steps. First, it is shown that the variance of the ground state energy is of the order of $|\Lambda|$. Then, the bound on the overlap follows from a relation between the variance and the overlap that is essentially a consequence of Gaussian integration by parts. 
Similar results for the model with external field were proved with different methods by Chen in \cite{chen_SK},  and for the spherical version of the model by Chen \& Sen in \cite{chen_sen_spherical}.

The main motivation of the present paper is to pinpoint direct causes of absence of disorder chaos in finite dimension, namely the presence of a positive density of coupling features that stabilize the ground state. 
We provide three proofs of Theorem \ref{thm: main}. 
The two proofs in Section \ref{sect: ss} and \ref{sect: swamp} rely on the presence of strong ferromagnetic couplings. The one in Section \ref{sect: ss} is simpler, but uses the assumption that $0$ is in the support. The proof in Section \ref{sect: swamp} relies on no other assumption than the continuity of the distribution.

Section \ref{sect: droplet} presents a different approach based on controlling the influence on the ground state of the coupling at a given edge. More precisely, we look at the critical droplet at an edge $e$, i.e., the set of vertices that flips when a coupling at $e$ is sent to either $+\infty$ or $-\infty$. 
It is known, see for example \cite{ANS_2019}, that the size of the critical droplet is intimately related to the number of ground states in the infinite-volume limit. This is still an important open question to be resolved related to the existence and the nature of the spin glass phase transition in finite dimension. 
The situation is more tractable for the model on trees, see \cite{baumler}, and on the half-plane, see \cite{ADNS_2010, arguin_damron}.
We expect that, at least for $d=2$, the critical droplets of all edges have finite size (uniformly in $\Lambda$), in which case a stronger version of Theorem \ref{thm: main} should hold:
\begin{conj}
Consider the Hamiltonian \eqref{eqn: Hamilton} at $d=2$. For any $\delta>0$ there exists $A=A(\delta)$ and $t_0=t_0(\delta)$ (independent of $\Lambda$) such that $\PP(A)>1-\delta$ uniformly in $\Lambda$, and
$$
\text{on $A$, }\ Q_\Lambda(\sigma(J+t\e),\sigma(J+t\e'))>1-\delta \qquad \text{for all $t\leq t_0$.}
$$
\end{conj}
In words, as the perturbation is turned on, all but a set of vertices of small density remain unchanged. 
If true, then this implies that the variance of the difference of ground state energies goes like the volume of $\Lambda$, see Theorem 1.5 in \cite{ANS_2019}.
This would likely give an approach to prove uniqueness of the ground state in the infinite volume by implementing a strategy similar to the one of Aizenman \& Wehr for the random field Ising model \cite{aizenman-wehr_1990}. 
Another interesting result in $d=2$ that may be relevant to the nature of the ground states is that the satisfied edges, where $\sigma_x\sigma_y=\text{sgn }J_{xy}$, do not percolate in an infinite-volume ground state as shown by Berger \& Tessler in \cite{berger_tessler}.

Our results extend to positive temperature, where $\sigma$ is no longer the ground state, but rather independent random configurations $\sigma, \sigma'$. The configuration $\sigma$ is sampled with probability proportional to the Gibbs weight $e^{-\beta H_{\Lambda, J + t\varepsilon}(\sigma)}$ for fixed realizations of $J$ and $\varepsilon$, and $\sigma'$ is chosen analogously (with $\varepsilon'$ replacing $\varepsilon$). The distribution of such a pair is denoted by $\langle \cdot \rangle_t$. Note that, when $t = 0$, the states $\sigma, \sigma'$ are independently sampled from the same distribution.

For simplicity, we prove the positive-temperature result only under the simplifying assumption that $0$ is in the support of $J$ (as in the zero-temperature proof of Section \ref{sect: ss}), though this assumption can be removed by techniques similar to those of Section \ref{sect: swamp}.
\begin{thm}[Positive Temperature]\label{thm: positive}
	Assume $0$ is in the support of $J$, and let $\langle \cdot \rangle_t$ be defined as above. For any $\delta > 0$ there exists $t_0 = t_0(\delta)$ such that
	\[\langle Q_\Lambda(\sigma,\sigma') \rangle_t >(1-\delta) \delta' \PP(A)\\ , \text{ for any $t<t_0(\delta)$},\]
	where $A$ is some explicit event with $\PP(A) > 0$ uniformly in $\Lambda$ and $\delta' > 0$ is an explicit constant depending on the distribution of $J$.
	
\end{thm}

\noindent {\bf Notation.}
We write $B(x,n)$ for the set of vertices whose $\ell^\infty$-distance to $v$ is less or equal to $n$. 
In other words, $B(x,n)$ is a box centered at $v$ of sidelength $2n+1$. 
We also write $\|\cdot\|_1$ for the $\ell^1$-norm on $\Z^d$. 
For the sake of conciseness, we will often use the following notation for the product $\sigma_x\sigma_y$:
$$
\text{for $e=\{x,y\}$, }\sigma_e=\sigma_x\sigma_y\ .
$$

\subsection{Method of Proof}
\label{sect: method}
The three proofs of the theorem are based on the following idea.
For a given edge $e$, we find a subset $A_e\subseteq \R^{\lpl^*}$ of realizations of couplings such that
\begin{itemize}
\item $A_e$ is open and depends on a finite number of couplings uniformly in $\Lambda$ and $e$ (i.e., it is an open cylinder set);
\item $\sigma_e(J)$ is constant on $A_e$.
\end{itemize}
To prove the theorem, we first write the overlap in terms of the event $A_e$
\begin{equation}
\begin{aligned}
\E[Q_\Lambda(\sigma(J+t\e),\sigma(J+t\e'))]&=\frac{1}{|\Lambda^*|}\sum_{e\in \Lambda^*} \E[\sigma_e(J+t\e)\sigma_e(J+t\e'));  J \in A_e]\\
&+\frac{1}{|\Lambda^*|}\sum_{e\in \Lambda^*} \E[\sigma_e(J+t\e)\sigma_e(J+t\e')); J \in A_e^c]\ .
\end{aligned}
\end{equation}
We observe that, by conditioning on $J$, the independence of the perturbations $\e, \e'$ yields
$$
\E[\sigma_e(J+t\e)\sigma_e(J+t\e'));  J \in A_e^c]
=\E[(\E[\sigma_e(J+t\e)|J])^2\ ; J \in A_e^c]\geq 0\ .
$$
Therefore, this gives the lower bound
$$
\begin{aligned}
\E[Q_\Lambda(\sigma(J+t\e),\sigma(J+t\e'))]&\geq \frac{1}{|\Lambda^*|}\sum_{e\in \Lambda^*} \E[\sigma_e(J+t\e')\sigma_e(J+t\e'));  J \in A_e]\\
&=\frac{1}{|\Lambda^*|}\sum_{e\in \Lambda^*} \Big\{2\PP(\{\sigma_e(J+t\e)=\sigma_e(J+t\e')\} \cap  \{ J \in A_e\})- \PP(A_e)\Big\}\ .
\end{aligned}
$$

Since by assumption the ground state at $e$ is constant on $A_e$, the summand is larger than
$$
2\PP(  J \in \left[A_e\cap (A_e-t\e) \cap (A_e-t\e')\right])- \PP( A_e)\ ,
$$
where $A_e-t\e$ stands for the translate of $A_e$ by $t\e$.
Moreover, since we assume that $A_e$ is open and depends only on a finite number of couplings uniformly in $\Lambda$, for any $0<\delta<1$, there exists $t_0(\delta)$ independent of $\Lambda$ such that
$$
\PP( J \in \left[ A_e\cap (A_e-t\e) \cap (A_e-t\e') \right])>(1-\delta)\PP( A_e) \text{ for $t<t_0(\delta)$}\ .
$$
The conclusion of the theorem follows from this. 

\section{Proof of Theorem \ref{thm: main} using (anti-)ferromagnetic edges}
\subsection{Proof under assumptions on the support of $J$}
\label{sect: ss}
We first suppose that a neighborhood of $0$ is included in the support of the distribution of $J$.

Let $e=\{x,y\}$ be an edge. We write $A_{e}$ for the set of realizations of couplings
$$
A_e=\left\{ J: \ J_{xy}>\sum_{i=1}^{2d-1}|J_{xz_i}|\right\}
$$
where $z_i$, $i=1,\dots, 2d-1$, stands for the vertices neighboring $x$ other than $y$. 
Clearly, $A_e$ is an open set of $ \R^{\lpl^*}$ that depends on only finite number of couplings. Moreover, we have $\PP(A_e)>0$ uniformly in $\Lambda$ if $0$ is in the support of the distribution of $J$. Such edges were referred to as {\it super-satisfied} in \cite{newman-stein_2001, ADNS_2010}. 

\begin{proof}[Proof of Theorem \ref{thm: main}]
In view of Section \ref{sect: method}, it remains to show that the ground state is constant on $A_e$. 
In fact, it must be that $\sigma_e=+1$, otherwise the ground state property \eqref{eqn: GS prop} is violated for $B=\{x\}$.

\end{proof}

\subsection{Proof under no assumptions on $J$}
\label{sect: swamp}
We present a modified construction to prove Theorem \ref{thm: main} without assumptions on the support of $J$. We are no longer able to force the satisfaction status of a given $e$ in the ground state. Hence, we now construct an event $A_v$ for a fixed vertex $v$ on which, for suitably chosen $M$, many edges of $B(v; M)^*$ have stable satisfaction status. In other words, for a fixed vertex $v$, we show $\sigma_f(J+t\e) = \sigma_f(J)$ for ``most'' $f$ near $v$. The argument makes clear the role of the finite-dimensionality in the absence of disorder chaos: perturbation-induced changes in boundary conditions change the energy by at most the order of the boundary size, but this cannot flip order of the volume number edges.

Fix $M$ large and constant (to be chosen precisely). Choose some interval $I = (a, b)$ such that $\PP(J_{xy} \in I) > 0$ and such that $0 \notin [a,b]$ ; we write $\sign(I)$ for the common sign of the elements of $I$.  For $v$ such that $B(v;M) \subset \Lambda$, we define $A_v$ as follows:
\begin{equation*}
A_v = \left\{ J:\ J_f \in I, \text{for all } f \in B(v;M)^* \cup \partial B(v;M)\right\}\ .
\end{equation*}
Note that $A_v$ is an open subset of $ \R^{\lpl^*}$ that depends on a finite number of couplings.
On the event $A_v$, the contribution to the total energy from bonds within $B(v; M)^*$ is minimized by homogeneous ferromagnetic or antiferromagnetic configurations (depending on  the value of $\sign(I)$), which satisfy all bonds of $B(v;M)$. The contribution to the energy from bonds in $\partial B(v;M)$ is only of order $M^{d-1}$, and so it will follow that the restriction of $\sigma(J)$ to $B(v;M)$ must still satisfy a high density of bonds of $B(v;M)$. This is the content of the following proposition. 

\begin{prop}\label{prop:energyscale}
	Suppose $M$ is large enough that $(|a| \wedge |b|) M^d  > 2d(2M + 1)^{d-1} (|a| \vee |b|)$, and let $v \in \Lambda$ such that $B(v;M) \subseteq \Lambda$. Then on the event $A_v$, we have
	\[ \left| \left\{ f \in B(v;M)^*: J_f \sigma_f(J) > 0  \right\}\right| \geq (3/4) |B(v;M)^*|\ . \]  
\end{prop}
\begin{proof}
	Suppose for a contradiction that the above cardinality bound were false; in particular, $J_f \sigma_f(J) < 0$ for at least $M^d$ edges $f \in B(v;M)^*$ (note that $|B(v;M)^*| = d (M-1)M^{d-1}$). Define the modification $\overline \sigma$ of $\sigma(J)$ obtained by satisfying all bonds in $B(v;M)^*$ and leaving the configuration outside $B(v;M)$ unchanged:
	\[\overline \sigma_x = \begin{cases} \sigma_x(J)\ , \quad &x \notin B(v;M);\\
	(\sign(I))^{\|x\|_1}\ , \quad & x \in B(v;M).
	\end{cases} \]
	Since $\sigma(J)$ is the ground state, we have $H_{\Lambda, J}(\overline \sigma) - H_{\Lambda, J}(\sigma(J)) \geq 0$. Estimating this energy difference directly, we see
	\begin{align}
	H_{\Lambda, J}(\overline \sigma) - H_{\Lambda, J}(\sigma(J)) &= \left[\sum_{x, y \in B(v;M)}  + \sum_{\substack{x\in B(v;M) \\  y \in  \partial B(v;M)} } \right] J_{xy} (\overline \sigma_x \overline \sigma_y - \sigma_x(J) \sigma_y(J))\nonumber\\
	&\leq -2(|a| \wedge |b|) M^d  + 4d(2M + 1)^{d-1} (|a| \vee |b|)\ .\label{eq:enlowbd}
	\end{align}
	In estimating the first term above, we used the fact that on $A_v$ all edges in $B(v;M)^*$ are satisfied in $\overline \sigma$, but (by assumption) at least $M^d$ are unsatisfied in $\sigma(J)$. For the second term, we use the fact that $|\overline \sigma_f - \sigma(J)_f| \leq 2$ for each edge $f$.
\end{proof}

We now prove Theorem \ref{thm: main}.
\begin{proof}[Proof of Theorem \ref{thm: main}]
 Fix $M$ as in the statement of Proposition \ref{prop:energyscale}. We break $\Lambda$ up into blocks of sidelength $(2M+1)$ centered at vertices $v \in \Lambda$ (with some smaller boxes at the boundary if the sidelength of $\Lambda$ is not a multiple of $(2M+1)$ --- since there are $o(|\Lambda|)$ of these, we can disregard them in what follows). By Proposition \ref{prop:energyscale}, we have for all $t<t_0(\delta)$ for some $t_0(\delta)$ (independent of $\Lambda$),
\[\sum_{e \in B(v;M)^*}\PP(\{\sigma_e(J+t\e)=\sigma_e(J+t\e')\}\cap A_v) \geq (1-\delta)(3/4) |B(v;M)^*| \cdot\PP( A_v)\ . \]
The claim then follows similarly to Section \ref{sect: method}, since the number of boxes $B(v,M)$ is proportional to the size of $\Lambda$.
\end{proof}

\section{Proof of Theorem \ref{thm: main} using critical droplets}
\label{sect: droplet}
For a fixed edge $e=\{x,y\}$, a good measure of the sensitivity under perturbation of the ground state at $e$ is given by the set of vertices $\mathcal D_e=\mathcal D_e(J)$ containing either $x$ or $y$ with the lowest flip energy in the ground state $\sigma(J)$. In other words, 
$$
\sum_{b\in \partial \mathcal D_e} J_b\sigma_b(J)=\min_{\mathcal B:  e\in \partial \mathcal B }\Big\{\sum_{b\in \partial \mathcal B} J_b\sigma_b(J) \Big\}\ .
$$
The set $\mathcal D_e$ is referred to as the {\it critical droplet of the edge $e$}. More generally, we consider the spin configurations that minimize $H_{\Lambda, J}$ with a fixed configuration at $e$
$$
\sigma^{\pm, e}(J)=\argmin_{s\in \{-1,+1\}^{\Lambda}, s_e=\pm 1} H_{\Lambda, J}(s)\ .
$$
As for the ground state, the states $\sigma^{\pm, e}$ as functions of $J$ are well-defined on the open set $\R^{(\Lambda\cup \partial \Lambda)^*}\setminus \mathcal C$. 
Clearly, the ground state $\sigma$ is either $\sigma^{+, e}$ or $\sigma^{-, e}$. Moreover, the critical droplet $\mathcal D_e$ is exactly the set of vertices where $\sigma^{+,e}$ and $\sigma^{-,e}$ differ. Similarly as in \eqref{eqn: GS prop}, for any set of vertices $\mathcal B$ such that $e\notin \partial \mathcal B$, we must have 
\begin{equation}
\label{eqn: GS prop 2}
\sum_{b\in \partial \mathcal B}J_b\sigma^{\pm,e}_b(J)>0\ .
\end{equation}
The following elementary fact will be needed.
\begin{lem}
	\label{lem: droplet}
	Let $\mathcal D_e$ be the critical droplet of the edge $e=\{x,y\}$. Then $\mathcal D_e$ and $\mathcal D_e^c$ are a.s.~connected as subgraphs of $\Lambda$.
\end{lem}

\begin{proof}
Without loss of generality, we take $x\in \mathcal D_e$. Suppose $\mathcal D_e^c$ is not connected. Then it has at least one connected component, say $\mathcal B$, that does not contain $y$ nor $x$. In particular, $\partial\mathcal B$ does not contain $e$. 
But by definition of the droplet, the energy of the boundary in $\sigma^{+,e}$ and $\sigma^{-.e}$ are of opposite signs. In particular, this implies
$$
\sum_{b\in \partial \mathcal B}J_b\sigma^{+,e}(J)<0 \text{ or }\sum_{b\in \partial \mathcal B}J_b\sigma^{-,e}(J)<0,
$$ 
thereby contradicting Equation \eqref{eqn: GS prop 2}.
\end{proof}

In the next section, we explicitly construct an event with positive probability on which the critical droplet is of size one. 
In particular, this implies that the ground state is constant on some subset $A_e$ of that event, thereby providing another proof of Theorem \ref{thm: main}
In Section \ref{sect: droplet}, we show how this argument can be generalized on the event that the droplet is of finite size (uniformly in $\Lambda$) with positive probability.

\subsection{A Critical Droplet of Size One}
\label{sect: dropsizeone}
We first describe a construction which shows that $\partial \mathcal{D}_e$ can be of order one (in fact, of cardinality exactly one) with nonvanishing probability, based on the presence of locally ferromagnetic regions. Roughly speaking, a local region of sufficiently ferromagnetic bonds encircling $e$ causes nearby spins to strongly prefer to align, preventing the droplet from propagating outside this region. 

The construction requires that the common distribution of the $J_{xy}$'s has a support which is not too concentrated, and so we work under the following assumption:
\begin{assn} \label{assn: support} There is an $r \geq 0$ such that both $\PP(|J_e| \leq r) > 0$ and $\PP(|J_e| > d\cdot 3^d r) > 0$.
\end{assn}
(This obviously holds if $0$ is in the support, for example.)
The above assumption allows us to show a particularly strong form of bounded droplet size, as in the following proposition. 
\begin{prop}\label{prop: cage}
	Suppose that the distribution of $J$'s satisfies Assumption \ref{assn: support}. 
	Let $e = \{x,y\} \in \Lambda^*$ be an edge such that both $x,y$ are at a distance at least $2$ from the boundary.
	There exists $c > 0$, uniform in $\Lambda$ and in the choice of $e$, such that $\PP\left( \mathcal{D}_e = \{x\} \right) \geq c$.
\end{prop}
\begin{figure}[h]
\includegraphics[height=3.5cm]{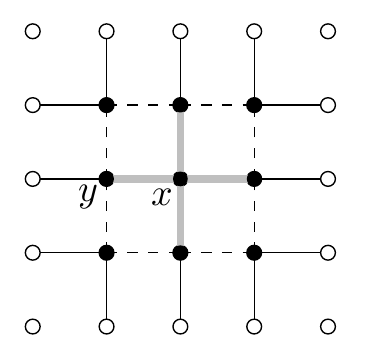}
\caption{A depiction of the defining conditions of the event $A_e$, $e=\{x,y\}$. The vertices in $B(x,1)$ are the black circles. Edges satisfying Condition (1) are the dotted lines. The ones satisfying Condition (2) are the the black lines, and the ones under Condition (3) are in grey. 
}
\label{fig: flex}
\end{figure}
The set $A_e$ of coupling realizations needed to prove Theorem \ref{thm: main}, as outlined in Section \ref{sect: method}, is defined by the following three conditions: let $v\in B(x;1)\setminus \{x\}$,
 
\begin{enumerate}
	\item if $w  \notin B(x;1)\setminus \{x\}$, then $|J_{vw}| < r;$
	\item if $w  \in  B(x;1)\setminus \{x\}$, then $|J_{vw}| > d\cdot 3^dr,$ and all such $J_{vw}$ have the same sign;
	\item $\sum_{w: \{x,w\}\in \Lambda^*} J_{xw}>0$. (If $J$ is negative with probability one, then take $<0$.)
\end{enumerate}

By construction, the set $A_e$ is 
an open set that depends on a finite number of couplings. Moreover,  its probability is positive and independent of $e$ by Assumption \ref{assn: support}.
We prove that $A_e\subset \{\mathcal{D}_e= \{x\}\}$, thereby implying Proposition \ref{prop: cage}. 
\begin{proof}[Proof of Proposition \ref{prop: cage}]
 We prove a stronger claim: if $u \in B(x;1) \setminus \{x\}$, then $\sigma^{+,e}_u = \sigma^{-,e}_u$. 
 This immediately implies Proposition~\ref{prop: cage}: $\mathcal{D}_e$ is connected (see Lemma \ref{lem: droplet}) and contains $x$ but no neighbor of $x$, so it must be $\{x\}$.
 
 Let $\mathcal{S} \subset \{-1, +1\}^\Lambda$ be the set of spin configurations $s$ for which all edges between neighbors of $x$ are satisfied, i.e., if $\{u,v\} \in \Lambda^*$ and $u,v\in  B(x;1) \setminus \{x\}$, then   $J_{uv} s_u s_v > 0$. 
 As a first step we show that $\sigma^{\pm,e}\in \mathcal S$ on $A_e$.
 Suppose it is not the case for $\sigma^{+,e}$, i.e., there exist $u_0,v_0$ as in the statement of the proposition, $J_{u_0 v_0}\sigma^{+,e}_{u_0}\sigma^{+,e}_{v_0} < 0$.
 We construct another spin configuration, denoted $\eta$, satisfying $\eta_e = +1$ for which $J_{uv} \eta_{u} \eta_v > 0$ for all $u,v$. We show that $\eta$ has lower energy than $\sigma^{+,e}$, contradicting the definition of $\sigma^{+,e}$.

 We set $\eta_x = +1$ and choose the value of $\eta$ at other sites as follows. If the common sign of the $J_{vw}$'s in item (2) of the definition of $A_e$ is positive, then we let $\eta_u = \sigma^{+,e}_u$ when $u \notin B(x;1) \setminus \{x\}$; when $u \in B(x;1) \setminus \{x\}$, we set $\sigma_u = +1$. If the common sign of the $J_{vw}$'s in item (2) is negative, we again let $\eta_u = \sigma_u$ when $u \notin  B(x;1) \setminus \{x\}$, but now set $\eta_u = (-1)^{1+\|u-x\|_1}$ when $u \in B(x;1) \setminus \{x\}$. 
	
	We claim that $H_{\Lambda, J}(\sigma^{+,e}) > H_{\Lambda, J}(\eta)$. Indeed, because $\sigma^{+,e}$ and $\eta$ may disagree only at $x$ or on $B(x;1) \setminus \{x\}$, 
	\begin{align}
	H_{\Lambda, J}(\sigma^{+,e}) - H_{\Lambda, J}(\eta)  &= \sum_{\substack{u, v:\, u \in  B(x;1) \setminus \{x\} }} J_{u v} (\eta_u \eta_v -  \sigma^{+,e}_u \sigma^{+,e}_v)\nonumber\\
	&= \sum_{\substack{u,v \in B(x;1) \setminus \{x\}} } J_{u v} (\eta_u \eta_v -  \sigma^{+,e}_u \sigma^{+,e}_v) + \sum_{\substack{u \in  B(x;1) \setminus \{x\} \\v \notin  B(x;1) \setminus \{x\} }} J_{u v} (\eta_u \eta_v -  \sigma^{+,e}_u \sigma^{+,e}_v)\label{eq: intermedsum}\\
	&\geq 2 |J_{u_0, v_0}| - d\cdot (2)(3^d)  \cdot \max_{\substack{u \in B(x;1) \setminus \{x\}, \\ v \notin B(x;1) \setminus \{x\}}} |J_{uv}| > 0\ .\nonumber
	\end{align}
	Here the first sum in \eqref{eq: intermedsum} is bounded by noting that each $J_{uv} \eta_u \eta_v$ term of that sum is positive, but at least the term $J_{u_0 v_0}\sigma^{+,e}_{u_0}\sigma^{+,e}_{v_0}$ is negative. The second sum in \eqref{eq: intermedsum} is controlled by lower-bounding each term by $- 2|J_{uv}|$; the final inequality comes from the definition of $A_e$.
This completes the contradiction and the proof in the case of $\sigma^{+,e}$. The proof in the case of $\sigma^{-,e}$ is similar.

Consider the bijection $\varphi: \mathcal{S} \to \mathcal{S}$ that flips the spin at $x$: $\varphi(s)_x = - s_x$ and $\varphi(s)_u = s_u$ for $u \neq x$. 
It remains to show that $\sigma^{-,e}=\varphi(\sigma^{+,e})$.
Observe that the map $\varphi$ maps spin configurations $s,s'$ such that $J_e s_e = J_e s'_e = 1$ to configurations $\varphi(s), \varphi(s')$ such that $J_e \varphi(s)_e = J_e \varphi(s')_e = -1$, and 
	\begin{equation}
	\label{eqn: phi id}
	H_{\Lambda, J}(s) - H_{\Lambda, J} (\varphi(s)) = H_{\Lambda, J}(s') - H_{\Lambda, J} (\varphi(s')).
	\end{equation}
On $A_e$, it must be that $s_x s_u$ is constant as $u$ ranges over neighbors of $x$ for $s\in \mathcal S$. This implies:
	\begin{equation} \label{eq: flipen} \text{if }s \in \mathcal{S},\quad  H_{\Lambda, J}(\varphi(s)) - H_{\Lambda, J}(s) =  2 \sum_{u:\{u,x\}\in \Lambda^*} J_{xu}  s_x s_u =  2 s_e \sum_{u} J_{xu} \ .\end{equation}

	Assume without loss of generality that $\sigma(J) = \sigma^{+,e}$. Suppose it were the case that $H(\sigma^{-,e}) < H(\varphi(\sigma^{+,e}))$ --- and in particular that $\sigma^{-,e} \neq \varphi(\sigma^{+,e})$. Then, it would also follow by \eqref{eqn: phi id} and \eqref{eq: flipen} that $H(\varphi(\sigma^{-,e})) < H(\sigma^{+,e}),$ contradicting the fact that $\sigma^{+,e}$ is the ground state. This completes the proof.
\end{proof}

To prove Theorem \ref{thm: main}, it remains to show that $\sigma_e(J)$ is constant on $A_e$.
\begin{proof}[Proof of Theorem \ref{thm: main}]
By Proposition \ref{prop: cage}, the critical droplet $\mathcal D_e$ is simply $\{x\}$ on $A_e$. 
In particular, the ground state is solely determined between $\sigma^{+,e}$ and $\sigma^{-,e}$ by the condition
$$
\sum_{u:\{u,x\}\in \Lambda^*} J_{xu}  \sigma_x \sigma_u>0\ .
$$
We know from the proof of Proposition \ref{prop: cage} that, on $A_e$, $\sigma^{\pm,e}_x\sigma^{\pm,e}_u$ is constant as $u$ ranges over the neighbors of $x$. 
Therefore the above condition is reduced to 
$$
\sigma_e\cdot\left(\sum_{u:\{u,x\}\in \Lambda^*} J_{xu}  \right)>0\ .
$$
Condition 3 and the above ensure that $\sigma_e$ is constant on $A_e$ as claimed.
\end{proof}

\subsection{A General Argument for Finite Critical Droplet}
\label{sect: general droplet}
In this section, we prove:
\begin{prop}
\label{prop: general droplet}
Let $A_e$ be an open set of $\R^{ (\Lambda\cup \partial \Lambda)^*}$ depending on a finite number of edges such that: $A_e\subset \{|\mathcal D_e|<K\}$ for some $K>0$ and $\PP(A_e)>0$ uniformly in $\Lambda$. Then there exists an open set $\widetilde A_e \subset A_e$ depending on a finite number of edges such that  $\PP(\widetilde A_e)>0$ uniformly in $\Lambda$, and the ground state is constant on $\widetilde A_e$. 
In particular, Theorem \ref{thm: main} holds for the event $\widetilde A_e$. 
\end{prop}
The fact that the theorem holds is again by the reasoning of Section \ref{sect: method}. 
It remains to prove that the ground state can be made constant on a subset of $A_e$. 
For that purpose, we define the {\it flexibility at the edge $e$}
$$
\mathcal F_e(J) =\left|H_{\Lambda,J}(\sigma^{+,e})-H_{\Lambda,J}(\sigma^{-,e})\right|=2\sum_{\substack{ b = \{u,v\} :\\  u \in \mathcal D_e(J), v\in \partial \mathcal D_e(J) }} J_b\sigma_b(J)>0 \ .
$$
In other words, the flexibility is the minimal energy of all surfaces  going through $e$. 
The definition given above makes sense  whenever $J$ is not in the critical set $\mathcal C$. 
However, it can be extended to a continuous function on all of $\R^{\Lambda^*}$.
\begin{lem}\label{lem: flexderiv}
The flexibility $\mathcal F_e(J)$ extends to continuous function on $\R^{ (\Lambda\cup \partial \Lambda)^*}$. More precisely, it is a piecewise affine function of $J$ with 
$$
\frac{\partial \mathcal F_e(J)}{\partial J_b}=
\begin{cases}
2\sigma_b(J) \ &\text{ if $b\in \mathcal D_e(J)$},\\
0 \ &\text{ else,}
\end{cases}
$$
whenever $J\in \R^{(\Lambda\cup \partial \Lambda)^*}\setminus \mathcal C$. 
\end{lem}

\begin{proof}
Note that by definition we have
$$
H_{\Lambda, J}(\sigma^{+,e})=\min_{s\in \{-1,+1\}^{\Lambda}: s_e=+1} H_{\Lambda, J}(s)\qquad H_{\Lambda, J}(\sigma^{-,e})=\min_{s\in \{-1,+1\}^{\Lambda}: s_e=-1} H_{\Lambda, J}(s)
$$
Clearly, the function $J\mapsto H_{\Lambda, J}(s)$ is a continuous function of $J$ for a fixed $s$.
Therefore the minimum of such functions over a finitely many values of $s$ is itself a continuous function.
This shows that $J\mapsto \mathcal F_e(J)$ extends to a continuous function on $\R^{ (\Lambda\cup \partial \Lambda)^*}$.

For the derivatives, observe that $\sigma^{\pm ,e}(J)$ are locally constant on the complement of $\mathcal C$. 
Therefore if the ground state is given by $\sigma^{+ ,e}(J)$, say, we have
$$
\mathcal F_e(J)=H_{\Lambda, J}(\sigma^{- ,e}(J))-H_{\Lambda, J}(\sigma^{+ ,e}(J))
=2\sum_{{\substack{ b = \{u,v\} :\\  u \in \mathcal D_e(J), v\in \partial \mathcal D_e(J) }} } J_b \sigma_b(J)\ .
$$
The claim on the derivatives is then obvious.
\end{proof}
By putting the coupling $J_e$ apart, the flexibility can also be written as
\begin{equation}
\label{eqn: critical flex}
\mathcal F_e(J)=2|J_e-\mathcal C_e(J)|\ ,
\end{equation}
where $\mathcal C_e$ does not depend on $J_e$ as a function of $J$. 
In particular, $\mathcal C_e(J)$ seen as a random variable is independent of $J_e$. 
\begin{proof}[Proof of Proposition \ref{prop: general droplet}]
Consider, for some $\delta>0$ (to be fixed later), the set
$$
\widetilde A_e=A_e\cap \{\mathcal F_e(J)>\delta\}\ .
$$
The set $\widetilde A_e$ is open, since $\mathcal F_e$ is continuous, and depends only a finite number of edges since $\mathcal F_e$ is the flip energy of the droplet, the size of which is bounded by $K$ on $A_e$. 
Moreover, by Equation \eqref{eqn: critical flex}, the parameter $\delta$ can be taken small enough so that $\PP(\mathcal F_e>\delta)$ is arbitrarily close to $1$. 
This implies that for $\delta$ small enough $\PP( \widetilde A_e)>0$ uniformly in $\Lambda$. 
By considering a subset of $\widetilde A_e$ if necessary, we can assume without loss of generality that $\widetilde A_e$ is a cylinder set whose cross-section is a finite-dimensional ball in $K^*$ coordinates, having finite radius. 

We now prove that the flexibility is strictly positive on $\widetilde A_e$. 
For $J,J'\in A_e$, Lemma \ref{lem: flexderiv} implies the following representation
\begin{equation}
\label{eqn: flex taylor}
\F_e(J')=\F_e(J) +\int_{J\to J'} \nabla \F_s (\mathbf r)\cdot \rd \mathbf r\ .
\end{equation}
By the last paragraph, the above only depends on $K^*$ coordinates. In particular,
the difference can be bounded by
$$
|\F_e(J')-\F_e(J)|\leq \|\nabla \F_e\|\cdot \|J-J'\|_{\R^{K^*}}\ .
$$
This can be made smaller than $\delta/2$ (uniformly in $\Lambda$). This is because $\|\nabla \F_e\|\leq 2 K^*$  on $\widetilde A_e$ by the assumption on the size of the droplet on $A_e$ and Lemma \ref{lem: flexderiv}, and $\|J-J'\|_{\R^{K^*}}$ can be made as small as we wish by reducing the radius of $\widetilde A_e$. 
Together with \eqref{eqn: flex taylor}, this implies that $F_e(J)>0$ on $\widetilde A_e$. 

We now conclude that the ground state is constant on $\widetilde A_e$. 
Suppose it is not. Then there must exist $J$ and $J'$ in $\widetilde A_e$ such that $\sigma^{+,e}$ is the ground state at $J$ and $\sigma^{-,e}$ is the ground state at $J'$. 
In particular, this implies
$$
\F_e(J)=H_{\Lambda, J}(\sigma^{- ,e}(J))-H_{\Lambda, J}(\sigma^{+ ,e}(J)) \qquad \F_e(J')=H_{\Lambda, J'}(\sigma^{+,e}(J'))-H_{\Lambda, J}(\sigma^{- ,e}(J'))\ .
$$
By continuity of $\F_e(J)$, this implies that any path from $J$ to $J'$ in $\widetilde A_e$ will contain at least one point $J_0$ with $\F_e(J_0)=0$. 
This contradicts the previous result.
\end{proof}

	\section{Proof of the positive-temperature Theorem \ref{thm: positive}.}
	We define more explicitly the positive-temperature Gibbs specification $\langle \cdot \rangle_t$: given functions $f, g$ on $\{-1, +1\}^\Lambda$ and fixed joint disorder realization $J, \varepsilon, \varepsilon'$, we set
	\[\langle f(\sigma) g(\sigma') \rangle_t = \frac{1}{Z_t Z_t'} \sum_{\sigma, \sigma' \in \{-1, +1\}^\Lambda} \exp\left(-\beta\left[H_{\Lambda, J + t \varepsilon}(\sigma) +H_{\Lambda, J + t \varepsilon'}(\sigma)\right] \right) = \langle f(\sigma) \rangle_t \langle g(\sigma') \rangle_t \ , \]
	where the partition function $Z_t = \sum_{\sigma \in \{-1, +1\}^\Lambda} \exp(-H_{\Lambda, J + t \varepsilon})$  ($Z_t'$ is defined analogously, replacing $\varepsilon$ with $\varepsilon'$).
	We prove the theorem, similarly to the proof in Section \ref{sect: ss}, using the assumption on the support of $J$ to ``super-satisfy'' individual edges. To start, we again decompose to isolate the contribution to $Q_\Lambda$ from each edge:
	\begin{equation}\label{eqn: langlands}\E[\langle Q_\Lambda(\sigma,\sigma')\rangle_t] =\frac{1}{|\Lambda^*|}\sum_{e\in \Lambda^*} \E[\langle\sigma_e\sigma_e'\rangle_t ]\  ;\end{equation}
	Once we show that each term of the above is lower-bounded by $(1-\delta) \PP(A_e)$ for $t < t_0$ (for appropriate choices of $\PP(A_e)$ and $t_0$), the theorem will be proved.

	For simplicity, we assume that $\PP(J_e > 0) > 0$; the adaptations needed to treat the case $\PP(J_e < 0) = 1$ are straightforward. We make nearly the same definition  of $A_e$ as in Section \ref{sect: ss}, namely: $A_e=\left\{J_{xy}>\sum_{i=1}^{2d-1}|J_{xz_i}| + a \right\}$, where the $z_i$'s are the neighbors of $x$ other than $y$ and where $a > 0$ is chosen such that $\PP(A_e) > 0$. 
	We compute
	\begin{align}
	\E[\langle \sigma_e  \sigma_e' \rangle_t] = \E[\langle \sigma_e  \sigma_e \rangle_t; A_e] + \E[\langle \sigma_e  \sigma_e' \rangle_t; A_e^c]\ .\label{eqn: twotermspos}
	\end{align}
	Conditioning on $J$ and using independence, we see that the second term of \eqref{eqn: twotermspos} is positive:
	\[ \E[\langle \sigma_e  \sigma_e \rangle_t'; A_e^c] = \E[\langle \sigma_e \rangle_t \langle \sigma_e' \rangle_t; A_e^c] =  \E[ \E[\langle \sigma_e \rangle_t \mid J]^2; A_e^c] \geq 0\ .\]
	It thus suffices to lower-bound the first term of \eqref{eqn: twotermspos} for a particular $e = \{x,y\}$. 
	
	We recall from Section \ref{sect: dropsizeone} that  $\varphi: \{-1, +1\}^\Lambda \to \{-1, +1\}^\Lambda$ denotes the bijection that flips the spin at $x$: $\varphi(s)_x = - s_x$ and $\varphi(s)_u = s_u$ for $u \neq x$. Since $\varphi$ maps configurations $s$ such that $J_e s_e > 0$ to configurations such that $J_e \varphi(s)_e <  0$, we can use $\varphi$ to compare the contribution of such pairs $s, \varphi(s)$ to the expectation in \eqref{eqn: langlands}. On the event  $A_e \cap (A_e - t \varepsilon) \cap (A_e - t \varepsilon')$, we note as in Section \ref{sect: ss} that there is a nonnegative energy cost for failing to satisfy edge $e$. More explicitly: for each $s$ such that $s_e = 1$, both $H_{\Lambda, J + t \varepsilon}(s) \leq H_{\Lambda, J + t \varepsilon}(\varphi(s)) - a$ and $H_{\Lambda, J + t \varepsilon'}(s) \leq H_{\Lambda, J + t \varepsilon'}(\varphi(s)) - a$ whenever $0 \leq t < t_0(\delta)$ for an appropriate choice of $t_0(\delta)$.

	On this event, we estimate the thermal average $\langle\sigma_e\rangle_t$: 
	\begin{align*}
	\langle \sigma_e \rangle_t &= Z_t^{-1}\sum_{\sigma \in \{-1, +1\}^\Lambda} \sigma_e \exp({-\beta H_{\Lambda, J + t \varepsilon}(\sigma)})\\ &= Z_t^{-1}\sum_{\substack{\sigma \in \{-1, +1\}^\Lambda\\ \sigma_e = 1}} \exp({-\beta H_{\Lambda, J + t \varepsilon}(\sigma)}) - \exp({-\beta H_{\Lambda, J + t \varepsilon}(\varphi(\sigma))})\\
	&\geq (e^{\beta a} - 1)/2\quad \text{on $A_e \cap (A_e - t \varepsilon) \cap (A_e - t \varepsilon')$, for $0 \leq t < t_0$. }
	\end{align*}
	An identical argument shows the same lower bound for $\langle \sigma'_e\rangle_t$ under the same conditions, so
	\[\langle \sigma_e \sigma'_e\rangle_t = \langle \sigma_e \rangle_t \langle \sigma'_e\rangle_t \geq \delta' := ((e^{\beta a} - 1)/2)^2\ \text{ on $A_e \cap (A_e - t \varepsilon) \cap (A_e - t \varepsilon')$, for $0 \leq t < t_0$} ; \]
	applying this in the first term of \eqref{eqn: twotermspos} and inserting our estimates into \eqref{eqn: langlands} completes the proof of the theorem. \qed

\bibliographystyle{plain}

\bibliography{bib_EA_absence_DC.bib}

\begin{thebibliography}{10}

\bibitem{aizenman-wehr_1990}
M.~Aizenman and J.~Wehr.
\newblock Rounding effects of quenched randomness on first-order phase
  transitions.
\newblock {\em Comm. Math. Phys.}, 130(3):489--528, 1990.

\bibitem{arguin_damron}
L.-P. Arguin and M.~Damron.
\newblock On the number of ground states of the {E}dwards-{A}nderson spin glass
  model.
\newblock {\em Ann. Inst. H. Poincar\'{e} Probab. Statist.}, 50(1):28--62, 02
  2014.

\bibitem{ADNS_2010}
L.-P. Arguin, M.~Damron, C.~M. Newman, and D.~L. Stein.
\newblock Uniqueness of ground states for short-range spin glasses in the
  half-plane.
\newblock {\em Comm. Math. Phys.}, 300(3):641--657, 2010.

\bibitem{ANS_2019}
L.-P. Arguin, C.~M. Newman, and D.~L. Stein.
\newblock A relation between disorder chaos and incongruent states in spin
  glasses on $\mathbb z^d$.
\newblock {\em Communications in Mathematical Physics}, 367(3):1019--1043, May
  2019.

\bibitem{baumler}
J.~B{\"a}umler.
\newblock Uniqueness and non-uniqueness for spin-glass ground states on trees.
\newblock {\em Electron. J. Probab.}, 24:17 pp., 2019.

\bibitem{berger_tessler}
N.~Berger and R.~J. Tessler.
\newblock No percolation in low temperature spin glass 1.2.
\newblock {\em Electron. J. Probab.}, 22:Paper No. 88, 19, 2017.

\bibitem{bray-moore}
A.~J. Bray and M.~A. Moore.
\newblock Chaotic nature of the spin-glass phase.
\newblock {\em Phys. Rev. Lett.}, 58:57--60, Jan 1987.

\bibitem{chatterjee_2014}
S.~Chatterjee.
\newblock {\em Superconcentration and related topics}.
\newblock Springer Monographs in Mathematics. Springer, Cham, 2014.

\bibitem{chen_SK}
W.-K. Chen.
\newblock Disorder chaos in the {S}herrington-{K}irkpatrick model with external
  field.
\newblock {\em Ann. Probab.}, 41(5):3345--3391, 2013.

\bibitem{chen_sen_spherical}
W.-K. Chen and A.~Sen.
\newblock Parisi formula, disorder chaos and fluctuation for the ground state
  energy in the spherical mixed {$p$}-spin models.
\newblock {\em Comm. Math. Phys.}, 350(1):129--173, 2017.

\bibitem{fisher-huse}
D.~S. Fisher and D.~A. Huse.
\newblock Ordered phase of short-range ising spin-glasses.
\newblock {\em Phys. Rev. Lett.}, 56:1601--1604, Apr 1986.

\bibitem{newman-stein_2001}
C.~M. Newman and D.~L. Stein.
\newblock Are there incongruent ground states in 2{D} {E}dwards-{A}nderson spin
  glasses?
\newblock {\em Comm. Math. Phys.}, 224(1):205--218, 2001.
\newblock Dedicated to Joel L. Lebowitz.

\end{thebibliography}

\end{document}